\documentclass{amsart}

\newcommand{\lvt}{\left|\kern-1.35pt\left|\kern-1.3pt\left|}
\newcommand{\rvt}{\right|\kern-1.3pt\right|\kern-1.35pt\right|}

\usepackage[foot]{amsaddr}
\usepackage[english]{babel}
\usepackage[pdftex,hyperfootnotes]{hyperref}
\usepackage[usenames,dvipsnames,svgnames,table,x11names]{xcolor}
\usepackage{pgfplots}
\usepackage{tikz}
\usepackage{tikz-3dplot}
\usetikzlibrary{calc,shadows,shapes.callouts,shapes.geometric,shapes.misc,positioning,patterns,decorations.pathmorphing,decorations.markings,decorations.fractals,decorations.pathreplacing,shadings,fadings}

\usepackage{longtable}
\usepackage{enumerate}
\allowdisplaybreaks

\usepackage{fancybox}
\usepackage[utf8]{inputenc}

\usepackage{rotating,multirow,pdflscape}
\usepackage{amssymb,latexsym,amsmath,amsthm}
\usepackage{mathrsfs}
\usepackage{mathtools,arydshln,mathdots}
\usepackage{bigints}
\usepackage{comment}

\makeatletter
\renewcommand*\env@matrix[1][*\c@MaxMatrixCols c]{%
 \hskip -\arraycolsep
 \let\@ifnextchar\new@ifnextchar
 \array{#1}}
\makeatother

\mathtoolsset{centercolon}
\usepackage[table]{xcolor}
\usepackage[toc,page]{appendix}

\usepackage{tocvsec2}

\usepackage{Baskervaldx}
\usepackage{newtxmath }

 \newtheorem{nexam}{\bf Example}

\newtheorem{teo}{Theorem}
\newtheorem{pro}{Proposition}

\newcommand{\ii}{\operatorname{i}}

\renewcommand{\d}{\operatorname{d}}
\newcommand{\Exp}[1]{\operatorname{e}^{#1}}

\renewcommand{\Re}{\operatorname{Re}}

\newcommand{\C}{\mathbb{C}}

\newcommand{\N}{\mathbb{N}}

\title[Matrix Jacobi Polynomials via Riemann--Hilbert problem]{Matrix Jacobi Biorthogonal Polynomials via Riemann--Hilbert problem}



\author[A Branquinho]{Amílcar Branquinho$^1$}
\address{$^1$Departamento de Matem\'atica,
Universidade de Coimbra, 3001-454 Coimbra, Portugal}
\email{ajplb@mat.uc.pt}

\author[A Foulqué]{Ana Foulqui\'e-Moreno$^2$}
\address{$^2$Departamento de Matemática, Universidade de Aveiro, 3810-193 Aveiro, Portugal}
\email{foulquie@ua.pt}

\author[A Fradi]{Assil Fradi$^3$}
\address{$^3$
Mathematical Physics Special Functions and Applications Laboratory, The Higher School of Sciences and Technology of Hammam Sousse, University of Sousse, Sousse 4002, Tunisia}
\email{assilfradi@ua.pt}

\author[M Mañas]{Manuel Mañas$^4$}
\address{$^4$Departamento de Física Teórica, Universidad Complutense de Madrid, 28040-Madrid, Spain \&
Instituto de Ciencias Matematicas (ICMAT), Campus de Cantoblanco UAM, 28049-Madrid, Spain}
\email{manuel.manas@ucm.es}

\thanks{$^1$Acknowledges Centro de Matemática da Universidade de Coimbra (CMUC) -- UID/MAT/00324/2020, funded by the Portuguese Government through FCT/MEC and co-funded by the European Regional Development Fund through the Partnership Agreement PT2020.}

\thanks{$^2$ and $^3$ Acknowledges CIDMA Center for Research and Development in Mathematics and Applications (University of Aveiro) and the Portuguese Foundation for Science and Technology (FCT) within project UID/MAT/04106/2020.}

\thanks{$^4$Thanks financial support from the Spanish ``Agencia Estatal de Investigación'' research project [PGC2018-096504-B-C33], \emph{Ortogonalidad y Aproximación: Teoría y Aplicaciones en Física Matemática}
and research project [PID2021- 122154NB-I00], \emph{Ortogonalidad y aproximación con aplicaciones en machine learning y teoría de la probabilidad}.}

\begin{document}

\maketitle


\begin{abstract}
We consider matrix orthogonal polynomials related to Jacobi type matrices of weights that can be defined in terms of a given matrix Pearson equation. Stating a Riemann--Hilbert problem
we can derive first and second order differential relations that these matrix orthogonal polynomials and the second kind functions associated to them verify.
For the corresponding matrix recurrence coefficients, non-Abelian extensions of a family of discrete Painlev\'e d-P$_{IV}$ equations are obtained for the three term recurrence relation coefficients.
\end{abstract}

\section{Introduction}

In this paper we deal with \emph{regular matrix of weights} $ W(z)$,~i.e., their moments
\begin{align*}
W_n := \frac{1} {2\pi\ii} \int_\gamma z^n W (z) \d z,
&&
n\in\N ,
\end{align*}
are such that
$\det \big[ W_{j+k} \big]_{j,k=0, \ldots n}\not = 0 $, $n \in \N:=\{0,1,\ldots\} $, and where the support of $W$, is a non self-intersecting smooth curve, $\gamma$, on the complex plane with two end points at $a,b \in \mathbb C$, and such that it intersects the circles $|z|=R$, $R \in \mathbb{R}^{+}$, once and only once (i.e., it can be taken as a determination curve for arg $z$).

A Riemann--Hilbert approach was used to discuss matrix biorthogonal polynomials of Hermite~\cite{BFM,CM,Cassatella_3,GIM} and Laguerre type~\cite{BFAM}.
In this paper, we focus on Jacobi~type examples:
We say that a \emph{matrix of weights} $W= 
\begin{bsmallmatrix} 
W^{(1,1)} & \cdots & W^{(1,N)} \\
\vdots & \ddots & \vdots \\
W^{(N,1)} & \cdots & W^{(N,N)} 
\end{bsmallmatrix}
\in \C^{ N\times N}$ with support $\gamma$ is of \emph{Jacobi type} if
the entries~$W^{(j, k)}$ of the matrix measure $W$ can be written~as
\begin{align}\label{eq:the_weights}
W^{(j, k)}(z)=\sum_{m \in I_{j, k}} \varphi_{m}(z) (a+z)^{\alpha_{m}}(b-z)^{\beta_m} \log ^{p_{m}} (a+z)\log^{q_m}(b-z), 
&&
z \in \gamma ,
\end{align} 
where $I_{j, k}$ denotes a finite set of indexes, $\Re(\alpha_{m})$, $\Re(\beta_m)>-1$, $p_{m}$, $q_m \in \mathbb{N}$, $a\neq b$ real numbers and $\varphi_{m}$ is
Hölder continuous,
bounded and non-vanishing on~$\gamma$.
We assume that the determination of logarithm and the powers are taken along~$\gamma$. We will request, in the development of the theory, that the functions~$\varphi_{m}$ have a holomorphic extension to the whole complex plane.

This definition includes the non scalar examples of Jacobi type weights given in the literature~\cite{nuevo,McCc,McAg,McAg2,AgPJ,SvA}, and as far as we know it was not been studied in all its generality.

In this work, for the sake of simplicity, the finite end points of the curve $\gamma$ is taken at the origin, $a=0$ and $b=1$ with no loss of generality, as a similar arguments apply for $a \neq 0$ or $b\neq1$. In~\cite{duran2004} different examples of Laguerre matrix weights for the matrix orthogonal polynomials on the real line are studied.

The subject of orthogonal polynomials covers a wide range of topics within mathematics, as well as its applications. Particularly when dealing with Jacobi polynomials, they have played an important role in mathematical analysis.
Their origins can be traced back to classical problems
such as electromagnetism, potential theory and many other fields. Particularly, the Legendre and Chebyshev polynomials played a significant role in the development of spectral methods for partial differential equations~\cite{deift,ismail}.

It was Krein~\cite{Krein1,Krein2} who first introduced the matrix extensions of scalar orthogonal polynomials. On this topic, there have been some relevant papers published afterward~\cite{D2,D3,D4,DP,DvA,geronimo,MS,SvA}, as well as more recent contributions, including for instance~\cite{bere,geronimo}.
Numerous findings have been made such as scattering problem resolution and matrix Favard theorem~\cite{nikishin}.
Later, it was proven that matrix orthogonal polynomials sometimes satisfy some properties as do the classical orthogonal polynomials, such as the scalar type Rodrigues' formula~\cite{duran20052,duran_1}.
Moreover, the last few years have seen the discovery of a plethora of families of orthogonal matrix polynomials that are eigenfunctions of certain fixed second order linear differential operators with a matrix coefficients independent of the orthogonal polynomial degree~\cite{borrego,duran_3,duran2004}.

In this work we apply the Riemann--Hilbert analysis to the study of families of polynomials and its second kind functions that are orthogonal with respect to Jacobi type matrices of weights coming from a matrix Pearson equation.
We are able to derive first and second order differential relations that these matrix orthogonal polynomials and the second kind functions associated to them verify.
For the corresponding matrix recurrence coefficients, non-Abelian extensions of a family of discrete Painlev\'e d-P$_{IV}$ equations are obtained for the three term recurrence relation coefficients.

The structure of this work is the following:
In section~\ref{sec:2}, we present the basic theory about matrix biorthogonal polynomials and state the left and right Riemann--Hilbert problems.
Using orthogonal polynomials and second kind functions, the unique solution for the Riemann--Hilbert problem is given. 
In section~\ref{sec:3}, we discuss the analytic properties of the constant jump fundamental matrix associated with a matrix of weight solution of a Pearson type equation of Jacobi type.
In section~\ref{sec:4}, a Riemann--Hilbert approach is taken to derive the first and second order differential equation from the structure matrix.
We show that these equations reduce to the scalar case when the commutativity is imposed.
In section~\ref{sec:5}, we find a discrete nonlinear relation for the recursion coefficients that can be considered an extension of discrete Painlev\'e~IV. This is accomplished by computing the explicit expression for the structure~matrix.

\section{Biorthogonality and Riemann--Hilbert problem} \label{sec:2}

Given a 
regular matrix of weights $W$,
we define \emph{sequences of matrix monic polynomials},
$\left\{ P_{n}^{\mathsf L} (z) \right\}_{n \in \mathbb{Z}_{+}}$ and $\left\{ P_{n}^{\mathsf R}(z)\right\}_{n \in \mathbb{N}}$, respectively left orthogonal and right orthogonal, were \linebreak
$\deg P_{n}^{\mathsf L} (z)=n$ and $\deg P_{n}^{\mathsf R} (z)=n$, $n \in \mathbb{N}$, by the conditions,
\begin{align} \label{eq:ortogonalidadL}
\frac 1 {2\pi\ii} \int_\gamma {P}_n^{\mathsf L} (z) W (z) z^k \d z
= \delta_{n,k} C_n^{-1} , &&
\frac 1 {2\pi\ii} \int_\gamma z^k W (z) {P}_n^{\mathsf R} (z) \d z
= \delta_{n,k} C_n^{-1} , 
\end{align}
for $k = 0, 1, \ldots , n $ and $n \in \mathbb{N}$,
where $C_n $ is a nonsingular matrix.

The matrix of weights $ W $ induces a sesquilinear form in the set of matrix 
polynomials $\C^{N\times N}[z]$ given by
\begin{align*}
\langle P,Q \rangle_W :=\frac 1 {2\pi\ii} \int_\gamma {P} (z) W (z) Q(z) \d z,
\end{align*}
for which $\big\{ P_n^{\mathsf L}(z)\big\}_{n\in\N}$ and $\big\{ P_n^{\mathsf R}(z)\big\}_{n\in\N}$ are biorthogonal 
\begin{align*}
\big\langle P_n^{\mathsf L}, {P}_m^{\mathsf R} \big\rangle_ W
& = \delta_{n,m} C_n^{-1}, & n , m & \in \mathbb N .
\end{align*}
As the polynomials are chosen to be monic, we can write
\begin{align*}
{P}_n^{\mathsf L} (z) &= I_N z^n + p_{\mathsf L,n}^1 z^{n-1} + p_{\mathsf L,n}^2 z^{n-2} + \cdots + p_{\mathsf L,n}^n , \\
{P}_n^{\mathsf R} (z) &= I_N z^n + p_{\mathsf R,n}^1 z^{n-1} + p_{\mathsf R,n}^2 z^{n-2} + \cdots + p_{\mathsf R,n}^n , 
\end{align*}
with matrix coefficients $p_{\mathsf L, n}^k , p_{\mathsf R, n}^k \in 
\C^{N\times N} $, $k = 0, \ldots, n $ and $n \in \mathbb N $ (imposing that $p_{\mathsf L,n}^0 = p_{\mathsf R,n}^0=I_N $, $n \in \mathbb N $). Here $I_N\in\C^{N\times N}$ denotes the identity matrix.

We define, for all $n \in \mathbb N$, the \emph{sequence of second kind matrix functions} by
\begin{align*}
Q^{\mathsf L}_n (z)
:= \frac1{ 2 \pi \ii} \int_\gamma \frac{P^{\mathsf L}_n (z')}{z'-z} { W (z')} \d z^\prime , 
 &&
{Q}_n^{\mathsf R} (z)
:= \frac{1}{2\pi \ii} \int_\gamma W (z') \frac{P^{\mathsf R}_{n} (z')}{z'-z} \d z' .
\end{align*} 
From the orthogonality conditions~\eqref{eq:ortogonalidadL} 
we have, for $n \in \N$, the fol\-low\-ing asymptotic expansions, as $|z|\to\infty$
\begin{align*} 
Q^{\mathsf L}_n (z) 
& = - C_n^{-1} \big( I_N z^{-n-1} + q_{\mathsf L,n}^1 z^{-n-2} + \cdots \big),
\\
Q^{\mathsf R}_n (z) 
& = - \big( I_N z^{-n-1} + q_{\mathsf R,n}^1 z^{-n-2} + \cdots \big)C_n^{-1} . 
\end{align*}
We gather together these objects in the matrix
\begin{align}
Y^{\mathsf L}_n (z) & 
 : = 
\left[ \begin{matrix}
 {P}^{\mathsf L}_{n} (z) & Q^{\mathsf L}_{n} (z) \\[.05cm]
-C_{n-1} {P}^{\mathsf L}_{n-1} (z) & -C_{n-1} Q^{\mathsf L}_{n-1} (z)
 \end{matrix} \right]
 ,
\label{eq:ynl}
 &&
{Y}^{\mathsf R}_n (z) 
 : = 
\left[ \begin{matrix}
P^{\mathsf R}_{n} (z) & - P^{\mathsf R}_{n-1} (z) C_{n-1} \\[.05cm]
{Q}^{\mathsf R}_{n} (z) & - {Q}^{\mathsf R}_{n-1} (z) C_{n-1}
 \end{matrix} \right] .
\end{align}
In terms of the transfer matrices, the three term recurrence relations for $P^{\mathsf L}_{n}$, $P^{\mathsf R}_{n}$ and $Q^{\mathsf L}_{n}$,~$Q^{\mathsf R}_{n}$ read,
\begin{align*}
Y_{n+1}^{\mathsf L} (z) & 
 = T^{\mathsf L}_n (z) Y^{\mathsf L}_n (z),&&
 Y_{n+1}^{\mathsf R} (z)
= Y^{\mathsf R}_n (z) T^{\mathsf R}_n (z) ,
 &&
n \in \mathbb N ,
\end{align*}
where,
\begin{align*}
 T^{\mathsf L}_n=\begin{bmatrix}
z I_N - \beta_n^{\mathsf L} & C_{n}^{-1} \\[.05cm]
-C_{n} & {0}_N
\end{bmatrix},
&& T^{\mathsf R}_n= \begin{bmatrix}
z I_N - \beta_n^{\mathsf R} & -C_{n} \\[.05cm]
C_{n}^{-1} & {0}_N
\end{bmatrix} ,
\end{align*}
with initial conditions,
$ {P}^{\mathsf L}_{-1} = {P}^{\mathsf R}_{-1} = 0_N$,
${P}^{\mathsf L}_{0} = {P}^{\mathsf R}_{0} = I_N $,
$Q_{-1}^{\mathsf L}(z) = Q_{-1}^{\mathsf R}(z) =-C_{-1}^{-1}$,
$Q_{0}^{\mathsf L}(z) = Q_{0}^{\mathsf R}(z)= S_{W}(z):=\frac{1}{2 \pi \mathrm{i}} \int_{\gamma} \frac{W\left(z^{\prime}\right)}{z^{\prime}-z} \operatorname{d} z^{\prime}$,
where $S_{W}(z)$ is the Stieltjes--Markov transformation and $\beta_n^{\mathsf R} := C_n \beta^{\mathsf L}_n C_n^{-1}$.
We also know, cf. for example~\cite{BFM}, that
\begin{align*}
\big( {Y}^{\mathsf L}_n (z)\big)^{-1} = 
\left[ \begin{matrix}
0_N&I_N\\[.05cm]
-I_N &0_N
 \end{matrix} \right]
 {Y}^{\mathsf R}_n(z)
\left[ \begin{matrix}
0_N&-I_N\\[.05cm]
I_N &0_N
 \end{matrix} \right]
 .
\end{align*}

Now, we state a theorem on Riemann--Hilbert problem for the Jacobi type weights.
\begin{teo} \label{teo:LRHP}
Given a regular Jacobi type matrix of weights $W$ with support on $\gamma$ we have the matrix function $Y^{\mathsf L}_n$ and $Y^{\mathsf R}_n$, defined 
by~\eqref{eq:ynl}
is, for each $n \in \N$, the unique solution of the following Riemann--Hilbert problems, which consists, respectively, in the determination of a~$2N \times 2 N$ complex matrix function such~that:

{\noindent}{\rm (RH1):} $ Y_n^{\mathsf L}$ and $Y_n^{\mathsf R}$ is holomorphic in $\C \setminus \gamma $.

{\noindent}{\rm (RH2):} Satisfies the jump condition
\begin{align*} 
\big( Y^{\mathsf L}_n (z) \big)_+
=
\big( Y^{\mathsf L}_n (z) \big)_- \,\left[ \begin{matrix}I_N & W (z) \\ {0}_N & I_N \end{matrix} \right], 
 &&
\big( Y^{\mathsf R}_n (z) \big)_+ 
=
\left[ \begin{matrix}
 I_N & {0}_N \\ 
 W (z) & I_N 
 \end{matrix} \right] \big( Y^{\mathsf R}_n (z) \big)_- ,
 && z \in \gamma .
\end{align*}
{\noindent}{\rm (RH3):} Has the following asymptotic behavior,
as $|z| \to \infty$
\begin{align*} 
 Y_n^{\mathsf L} (z)
 =
\left(I_{2N}+\operatorname{O}({1}/{z})\right)\left[ \begin{matrix}z^{n}I_N & 0_N \\ 0_N & z^{-n}I_N \end{matrix} \right] , 
&&
Y_n^{\mathsf R} (z) = 
\left[ \begin{matrix}
I_N z^n & {0}_N \\ 
{0}_N & I_N z^{-n} 
 \end{matrix} \right]
\Big( I_N + \operatorname{O}({1}/{z}) \Big) .
\end{align*}
{\noindent}{\rm (RH4):}
$
Y^{\mathsf L}_n (z) 
= \left[ \begin{matrix}
\operatorname{O} (1) & s^{\mathsf L}_{1}(z) \\[.1cm]
\operatorname{O} (1) & s^{\mathsf L}_{2}(z) 
 \end{matrix} \right] $,
 $
Y^{\mathsf R}_n (z) 
= 
\left[ \begin{matrix}
\operatorname{O} (1) & \operatorname{O} (1) \\ 
s^{\mathsf R}_1(z) & s^{\mathsf R}_2(z) 
 \end{matrix} \right] $, 
as 
$z \to 0$,
with \linebreak
$\displaystyle \lim_{z \to 0} z s^{\mathsf L}_j(z) = 0_{N}$
and
$\displaystyle \lim_{z \to 0} z s^{\mathsf R}_j (z) =0_N$,
$j=1,2$.

{\noindent}{\rm (RH5):}
$
Y_{n}^{\mathsf L}(z)=
\left[ \begin{matrix}
\operatorname{O}(1) & r_{1}^{\mathsf L}(z) \\[.1cm]
\operatorname{O}(1) & r_{2}^{\mathsf L}(z)
 \end{matrix} \right]$,
$
Y_{n}^{\mathsf L}(z)
=\left[\begin{matrix}
\operatorname{O}(1) & \operatorname{O}(1) \\
r_{1}^{\mathsf R}(z) & r_{2}^{\mathsf R}(z)
\end{matrix}\right]
$,
as 
$z \to 1$,
with \linebreak
$\displaystyle \lim_{z \to 1} (1-z) r_{j}^{\mathsf L}(z)=0_{N}$
and
$\displaystyle \lim_{z \to 1} (1-z) r_{j}^{\mathsf R}(z)=0_{N}$,
$j=1,2$. 
The $s^{\mathsf L}_{i}$, $s^{\mathsf R}_{i}$ (respectively, $r^{\mathsf L}_{i}$ and $r^{\mathsf R}_{i}$) could be replaced by $\operatorname{o}({1}/{z})$, as $z\to 0$ (respectively, $\operatorname{o}({1}/({1-z}))$, as $z\to 1$).
The $ \operatorname{O} $ and $\operatorname{o}$ conditions are understood entry-wise.
\end{teo}

\begin{proof}
A very similar proof can be found in~\cite{BFM} and~\cite{BFAM}.
\end{proof}

\section{Fundamental matrices} \label{sec:3}

\subsection{Pearson equation}

Here we consider matrix of weights, $W$, satisfying a matrix Pearson type equation
\begin{align} \label{eq:pearsonjacobi}
z (1-z) W^\prime (z) = h^{\mathsf L} (z) W (z) + W (z) h^{\mathsf R} (z) ,
\end{align}
with entire matrix functions $h^{\mathsf L} $, $ h^{\mathsf R} $.
If we take a matrix function $W^{\mathsf L}$ such that
\begin{align} \label{eq:pearsonjacobileft}
z (1-z) ( W^{\mathsf L} )^\prime (z) = h^{\mathsf L} (z) W^{\mathsf L} (z) ,
\end{align}
then there exists a matrix function $W^\mathsf R (z)$ such that $W (z) = W^{\mathsf L}(z) W^{\mathsf R} (z)$ with
\begin{align} \label{eq:pearsonjacobiright}
z (1-z) ( W^{\mathsf R} )^\prime (z) = W^{\mathsf R} (z) h^{\mathsf R} (z) .
\end{align}
The reciprocal is also true.

The solution of~\eqref{eq:pearsonjacobileft} and~\eqref{eq:pearsonjacobiright} will have possibly branch points
at $0$ and $1$, cf.~\cite{wasow}. This means that the exists constant matrices, $\mathsf C_j^\mathsf L$, $\mathsf C_j^\mathsf R$, with $j =0,1$, such that
\begin{align}
\label{eq:salto0}
(W^{\mathsf L} (z))_- = (W^{\mathsf L} (z))_+ \mathsf C_0^\mathsf L, &&
(W^{\mathsf R} (z))_- = \mathsf C_0^\mathsf R (W^{\mathsf R} (z))_+, && \text{in } & (0,1) ,
 \\
\label{eq:salto1}
(W^{\mathsf L} (z))_- = (W^{\mathsf L} (z))_+ \mathsf C_1^\mathsf L, &&
(W^{\mathsf R} (z))_- = \mathsf C_1^\mathsf R (W^{\mathsf R} (z))_+, && \text{in } & (1,+\infty) .
\end{align}
We introduce, the \emph{constant jump fundamental matrices}, for $n \in \mathbb N$,
\begin{align}
\label{eq:zn1}
Z_n^{\mathsf L}(z) 
& : = Y^{\mathsf L}_n (z) 
\left[ \begin{matrix} 
W^{\mathsf L} (z) & 0_N \\ 
0_N & ( W^{\mathsf R} (z))^{-1} 
 \end{matrix} \right] ,
 &&
{Z}^{\mathsf R}_n (z) : =
\left[ \begin{matrix} 
W^{\mathsf R} (z) & 0_N \\ 
0_N & ( W^{\mathsf L} (z))^{-1} 
 \end{matrix} \right]
{Y}^{\mathsf R}_n (z) .
\end{align}
The constant jump fundamental matrices~$Z^{\mathsf L}_n$ and $Z^\mathsf R_n$ satisfy, for each $n \in \N$, the fol\-low\-ing properties: 

\begin{enumerate}

\item
Are holomorphic on $\C \setminus [0,+\infty)$.
\item
Present the fol\-low\-ing \emph{constant jump condition} on $(0,1)$
\begin{align*}
\big( Z^{\mathsf L}_n (z) \big)_+ &
= \big( Z^{\mathsf L}_n (z) \big)_- 
\left[\begin{matrix} 
\mathsf C_0^{\mathsf L} & \mathsf C_0^{\mathsf L}
 \\[.1cm]
0_N & I_N
\end{matrix} \right], 
 &&
\big( {Z}^{\mathsf R}_n (z) \big)_+ 
=
\left[\begin{matrix} 
I_N & {0}_N
 \\[.1cm]
\mathsf C_0^{\mathsf R} & \mathsf C_0^{\mathsf R}
\end{matrix} \right]
\big( {Z}^{\mathsf R}_n (z) \big)_- .
\end{align*}

\item
Present the fol\-low\-ing \emph{constant jump condition} on $(1,+\infty)$
\begin{align*}
\big( Z^{\mathsf L}_n (z) \big)_+ 
& = 
\big( Z^{\mathsf L}_n (z) \big)_- 
\left[\begin{matrix} 
\mathsf C_1^{\mathsf L} & 0_N
 \\[.1cm]
0_N & \mathsf C_1^{\mathsf R}
\end{matrix} \right],
 &&
\big( {Z}^{\mathsf R}_n (z) \big)_+
=
\left[\begin{matrix} 
\mathsf C_1^{\mathsf R} & {0}_N
 \\[.1cm]
0_N & \mathsf C_1^{\mathsf L}
\end{matrix} \right]
\big( {Z}^{\mathsf R}_n (z) \big)_- .
\end{align*}
\end{enumerate}

Now, we will explicit the constant jump matrix 
in the special case when we have the following decompositions for the matrix of weights,
$W (z) = W^{\mathsf L} (z) W^{\mathsf R} (z)$, with:
\begin{align}\label{eq:partial_Pearson1}
z \left(W^\mathsf L\right)^{\prime}(z) = \tilde h^{\mathsf L}(z) W^{\mathsf L}(z), &&
(1-z) \left(W^\mathsf R\right)^{\prime}(z)
 = W^{\mathsf R}(z) \tilde h^{\mathsf R}(z) ,
\end{align}
where $h^{\mathsf L}$ and $h^{\mathsf R}$ are entire functions.
Therefore, the matrix
$ W (z) =W^\mathsf L (z) W^\mathsf R (z)$
is such~that,
\begin{align*}
z(1-z) W^{\prime}(z)= {h}^{\mathsf L} (z) W(z)+ W(z) {h}^{\mathsf R} (z),
\end{align*} 
where ${h}^{\mathsf L}(z)=(1-z) \tilde h^{\mathsf L} (z)$ and ${h}^{\mathsf R}(z)=z \tilde h^{\mathsf R} (z)$. 

General solutions $W^\mathsf L$ and $W^\mathsf R$ of~\eqref{eq:partial_Pearson1} are given explicitly~(cf.~for example~\cite{wasow}) by
\begin{align}
\label{eq:casoparticular}
 W^{\mathsf L}(z) = H^{\mathsf L}(z) z^{\alpha} W_0^{\mathsf L},&& W^{\mathsf R}(z) =W_0^{\mathsf R} (1-z)^{\beta} H^{\mathsf R}(z),
\end{align}
where 
$
H^{\mathsf L}(z)$ and $H^{\mathsf R}(z)$ are entire and nonsingular matrix functions, and $\alpha, \beta$ are constant matrices, as well as~$W_0^{\mathsf L}$ and $W_0^{\mathsf R}$
are constant nonsingular matrices.

It is easy to see that $W$, within this decomposition, is a Jacobi type weight matrix defined by~\eqref{eq:the_weights}.
From~\eqref{eq:casoparticular}, the constant jump fundamental matrices~$Z^{\mathsf L}_n(z) $ and $Z^\mathsf R_n(z)$ have the following \emph{constant jump condition} on $(0,1)$
\begin{align*}
\big( Z^{\mathsf L}_n (z) \big)_+ &
= \big( Z^{\mathsf L}_n (z) \big)_- 
\left[\begin{matrix} 
{(W_0^{\mathsf L}})^{-1}\Exp{- 2 \pi \ii \alpha} W_0^{\mathsf L} & {(W_0^{\mathsf L}})^{-1}\Exp{- 2 \pi \ii \alpha} W_0^{\mathsf L}
 \\[.1cm]
0_N & I_N
\end{matrix} \right], \\
\big( {Z}^{\mathsf R}_n (z) \big)_+ &
=
\left[\begin{matrix} 
I_N & {0}_N
 \\[.1cm]
I_N & ({W_0^{\mathsf L}})^{-1}\Exp{2 \pi i \alpha} W_0^{\mathsf L}
\end{matrix} \right]
\big( {Z}^{\mathsf R}_n (z) \big)_- ,
\end{align*}
as well as, the \emph{constant jump condition} on $(1,+\infty)$
\begin{align*}
\big( Z^{\mathsf L}_n (z) \big)_+ &
= \big( Z^{\mathsf L}_n (z) \big)_- 
\left[\begin{matrix} 
{(W_0^{\mathsf L}})^{-1}\Exp{- 2 \pi \ii \alpha} W_0^{\mathsf L} & 0_N
 \\[.1cm]
0_N & W_0^{\mathsf R}\Exp{2 \pi i \beta} ({W_0^{\mathsf R}})^{-1}
\end{matrix} \right],\\
\big( {Z}^{\mathsf R}_n (z) \big)_+ &
=
\left[\begin{matrix} 
W_0^{\mathsf R}\Exp{-2 \pi i \beta} ({W_0^{\mathsf R}})^{-1}& {0}_N
 \\[.1cm]
0_N & \left({W_0^{\mathsf L}}\right)^{-1}\Exp{2 \pi \ii \alpha} W_0^{\mathsf L}
\end{matrix} \right]
\big( {Z}^{\mathsf R}_n (z) \big)_- .
\end{align*}
In fact, from the definition of $Z^{\mathsf L}_n (z)$ we have
\begin{align*}
\big(Z^{\mathsf L}_n (z) \big)_+ = \big( Y^{\mathsf L}_n (z) \big)_+ 
\left[ \begin{matrix}
(W^{\mathsf L} (z))_+ & 0_N \\ 
0_N & ( W^{\mathsf R} (z))_+^{-1} 
 \end{matrix} \right] , 
\end{align*}
and taking into account Theorem~\ref{teo:LRHP} we 
successively get
\begin{align*}
\big(Z^{\mathsf L}_n (z) \big)_+ &
= \big( Y^{\mathsf L}_n (z) \big)_-
\left[ \begin{matrix} 
I_N & (W^{\mathsf L} (z) W^{\mathsf R} (z))_+ \\ 
0_N & I_N
 \end{matrix} \right]
\left[ \begin{matrix} 
(W^{\mathsf L} (z))_+ & 0_N \\ 
0_N & ( W^{\mathsf R} (z))_+^{-1} 
 \end{matrix} \right]\\
&
 =
\resizebox{.815\hsize}{!}{$
\big( Y^{\mathsf L}_n (z) \big)_-
\left[ \begin{matrix} 
(W^{\mathsf L} (z))_- & 0_N \\ 
0_N & ( W^{\mathsf R} (z))_-^{-1} 
 \end{matrix} \right]
\left[ \begin{matrix} 
(W^{\mathsf L} (z))_-^{-1} & 0_N \\ 
0_N & ( W^{\mathsf R} (z))_- 
 \end{matrix} \right]
\left[ \begin{matrix} 
(W^{\mathsf L} (z))_+ & (W^{\mathsf L} (z))_+ \\ 
0_N & ( W^{\mathsf R} (z))_+^{-1} 
 \end{matrix} \right]$}
 \\
&
= 
\big(Z^{\mathsf L}_n (z) \big)_- 
\left[ \begin{matrix} 
(W^{\mathsf L} (z))_-^{-1} (W^{\mathsf L} (z))_+ &(W^{\mathsf L} (z))_-^{-1} (W^{\mathsf L} (z))_+ \\ 
0_N & W^{\mathsf R} (z)_- ( W^{\mathsf R} (z))_+^{-1} 
 \end{matrix} \right] .
\end{align*}
Similarly over $(1,+\infty)$ we have
\begin{align*}
\big(Z^{\mathsf L}_n (z) \big)_+ &
= \big( Y^{\mathsf L}_n (z) \big)_-
\left[ \begin{matrix} 
(W^{\mathsf L} (z))_+ & 0_N \\ 
0_N & ( W^{\mathsf R} (z))_+^{-1} 
 \end{matrix} \right]\\
&
 =
\resizebox{.815\hsize}{!}{$
 \big( Y^{\mathsf L}_n (z) \big)_-
\left[ \begin{matrix} 
(W^{\mathsf L} (z))_- & 0_N \\ 
0_N & ( W^{\mathsf R} (z))_-^{-1} 
 \end{matrix} \right]
\left[ \begin{matrix} 
(W^{\mathsf L} (z))_-^{-1} & 0_N \\ 
0_N & ( W^{\mathsf R} (z))_- 
 \end{matrix} \right]
\left[ \begin{matrix} 
(W^{\mathsf L} (z))_+ & 0_N \\ 
0_N & ( W^{\mathsf R} (z))_+^{-1} 
 \end{matrix} \right]
 $}
\\
&
= 
\big(Z^{\mathsf L}_n (z) \big)_- 
\left[ \begin{matrix} 
(W^{\mathsf L} (z))_-^{-1} (W^{\mathsf L} (z))_+ & 0_N \\ 
0_N & W^{\mathsf R} (z)_- ( W^{\mathsf R} (z))_+^{-1} 
 \end{matrix} \right] .
\end{align*}
To complete the proof we only have to see that
\begin{align*}
\left(W^{\mathsf L}\right)_-=H^{\mathsf L}\Exp{ 2 \pi \ii \alpha}z^\alpha W_0^{\mathsf L},
 &&
\left(W^{\mathsf R}\right)_-=W_0^{\mathsf R}\Exp{ 2 \pi \ii \alpha}(1-z)^\beta H^{\mathsf R} ,
\end{align*}
and then check that
\begin{align*}
Z^{\mathsf R}_n (z) =
\left[ \begin{matrix}
0 & -I_N \\ I_N & 0
 \end{matrix} \right]
(Z_n^{\mathsf L} (z))^{-1}
\left[ \begin{matrix}
0 & I_N \\ -I_N & 0
 \end{matrix} \right] ,
\end{align*}
which is a consequence of~\eqref{eq:zn1}
within the definition of $Y_n^\mathsf L$ and $Y_n^\mathsf R$, cf.~\eqref{eq:ynl}.

\subsection{Structure matrix and zero curvature formula}

In parallel to the matrices $Z^{\mathsf L}_n(z)$ and $Z^{\mathsf R}_n(z)$, for each factorization we introduce what we call \emph{structure matrices} given in terms of the \emph{left}, respectively \emph{right}, logarithmic derivatives by,
\begin{align} \label{eq:Mn}
M^{\mathsf L}_n (z) & 
 := \big(Z^{\mathsf L}_n\big)^{\prime} (z) \big(Z^{\mathsf L}_{n} (z)\big)^{-1},&
{M}^{\mathsf R}_n (z) & 
 := \big ({Z}^{\mathsf R}_{n} (z)\big)^{-1} \big(Z^{\mathsf R}_n\big)^{\prime} (z) .
\end{align}
It is not difficult to see that
\begin{align} \label{MnLR}
{M}^{\mathsf R}_{n}(z) &
 = -
\left[ \begin{matrix}
0 & -I_N \\ I_N & 0
 \end{matrix} \right]
M_n^{\mathsf L} (z)
\left[ \begin{matrix}
0 & I_N \\ -I_N & 0
 \end{matrix} \right] , &
n \in \mathbb N ,
\end{align}
as well as, the following properties hold~(cf.~\cite{BFM}):
\begin{enumerate}

\item
The transfer matrices satisfy
 \begin{align*}
 T^{\mathsf L}_n (z)Z_n^{\mathsf L}(z) & = Z^{\mathsf L}_{n+1}(z) ,& {Z}^{\mathsf R}_n(z){T}^{\mathsf R}_n (z) &=
 {Z}^{\mathsf R}_{n+1}(z), & n \in \mathbb N .
 \end{align*}

\item 
The zero curvature formulas holds for all $n \in \mathbb N$,
\begin{align*}
\left[ \begin{matrix}
I_N & 0_N\\
0_N & 0_N
 \end{matrix} \right]
 & 
= M^{\mathsf L}_{n+1} (z) T^{\mathsf L}_n(z) - T^{\mathsf L}_n (z) M^{\mathsf L}_{n} (z),
 \\
\left[ \begin{matrix}
I_N & 0_N \\
0_N & 0_N
 \end{matrix} \right]
 & =
T^{\mathsf R}_n (z) \, M^{\mathsf R}_{n+1} (z) 
- M^{\mathsf R}_{n} (z) T^{\mathsf R}_n (z) .
\end{align*}

\end{enumerate}

Now, we discuss the holomorphic properties of the structure matrices just introduced.
\begin{teo} \label{prop:mn}
Let $W$ be a regular Jacobi matrix weight that satisfies a Pearson type equation~\eqref{eq:pearsonjacobi} that admits a factorization $W (z) = W^{\mathsf L} (z) W^{\mathsf R} (z)$, where $W^{\mathsf L}$ and $W^{\mathsf R}$ satisfies~\eqref{eq:pearsonjacobileft} and~\eqref{eq:pearsonjacobiright}.
Then, the structure matrices~$M^{\mathsf L}_n(z) $ and $M^\mathsf R_n(z)$ are, for each $n \in \N$, meromorphic on~$\C $, with singularities located at $z=0$ and $z=1$, which happens to be a removable singularity or a simple~pole.
\end{teo}

\begin{proof}
Let us prove the statement for $M^{\mathsf L}_n(z) $.
The matrix function $M^{\mathsf L}_n(z) $ is holomorphic in $ \C \setminus
[0,+\infty)$ by definition,~cf.~\eqref{eq:Mn}.
Due to the fact that $Z^{\mathsf L}_n(z)$ has a constant jump on
$(0,1) \cup (1,+\infty)$,~cf.~\eqref{eq:salto0} and~\eqref{eq:salto1}, the matrix function $\displaystyle \big(Z^{\mathsf L}_n\big)^{\prime} $ has the same constant jump on $(0,1) \cup (1,+\infty)$, so that the matrix $M^{\mathsf L}_n(z) $ has no jump on $(0,1) \cup (1,+\infty)$, and it follows that
at $z=0$ and $z=1$, $M^{\mathsf L}_n(z) $ has an isolated singularity. 

From~\eqref{eq:zn1} and~\eqref{eq:Mn} it holds
\begin{align}
\label{eq:numerar}
M^{\mathsf L}_n (z)
&= \big(Z^{\mathsf L}_n\big)^{\prime} (z)\big({Z^{\mathsf L}_n} (z) \big)^{-1}
 \\
\nonumber
& = \big(Y^{\mathsf L}_n\big)^{\prime} (z)\big( {Y^{\mathsf L}_n} (z) \big)^{-1} 
+ \frac{1}{z(z-1)} Y^{\mathsf L}_n (z)
\left[ \begin{matrix}
h^{\mathsf L}(z) & 0_N\\
0_N & -h^{\mathsf R} (z)
 \end{matrix} \right]
\big( {Y^{\mathsf L}_n} (z)\big)^{-1}, 
\end{align}
where
$Y^{\mathsf L}_n$ is given in~\eqref{eq:ynl}.
Each entry of the matrix $Q^{\mathsf L}_{n} (z)$ is the Cauchy transform of certain function,~$f$, of type
\begin{align*}
f(z)=\sum_{j \in I} \varphi_{j}(z) z^{\alpha_{j}}(1-z)^{\beta_j} \log ^{p_{j}} (z) \log ^{q_{j}} (1-z),
\end{align*}
where $\varphi_{j} (z)$ is, for each $j \in I$, an entire function with $\Re(\alpha_{j})$, $\Re(\beta_{j})>-1$,
$ p_{j}$, $q_j \in \mathbb{N}$, and $I$ is a finite set of indices.
It's clear that
\begin{align*}
\lim_{z \to 0} z f(z)=0_N
&&
\text{and}
&&
\lim_{z \to 1}(1-z)f(z)=0_N
.
\end{align*}
By~\cite[\S 8.3-8.6]{gakhov} and~\cite{Muskhelishvili}, we deduce that the Cauchy transform of $f$ have the same properties:
\begin{align} \label{demon1}
\lim_{z \to 0} 
z \int_0^1 \frac{f(t)}{t-z}\, \d t =0_N
&&
\text{and}
&&
\lim_{z \to 1}
(1-z) \int_0^1 \frac{f(t)}{t-z} \, \d t =0_N .
\end{align}
Now, we will prove that
\begin{align}\label{demon2}
\lim_{z \to 0}
z^2 \left( \int_0^1 \frac{f(t)}{t-z}\, \d t\right)^\prime=0_N
 && \text{and} &&
\lim_{z \to 1}
(1-z)^2\left( \int_0^1 \frac{f(t)}{t-z} \, \d t\right)^\prime=0_N
 .
\end{align}
In fact,
\begin{align*}
& z(1-z) \left( \int_0^1 \frac{f(t)}{t-z}\, \d t \right)^\prime 
 = \int_{0}^1 \frac{z(1-z) f(t)}{(t-z)^{2}} \operatorname{d} t
\\
 & \phantom{olaolaolaol}
=\int_{0}^1 \frac{(t-z)(t+z-1) f(t)}{(t-z)^{2}} \operatorname{d} t
+\int_{0}^1 \frac{t(1-t) f(t)}{(t-z)^{2}} \operatorname{d} t, \\
 & \phantom{olaolaolaol}
=
\int_{0}^1 \frac{t+z-1}{t-z}f(t) \operatorname{d} t-
\left.\frac{t(1-t) f(t)}{t-z}\right|_{0}^1+
\int_{0}^1 \frac{\left(t(1-t) f(t)\right)^{\prime}}{t-z} \operatorname{d} t .
\end{align*}
From the boundary conditions, the first term is zero and we get
\begin{align*}
z (1-z)\left( \int_0^1 \frac{f(t)}{t-z}\, \d t\right)^\prime
=-\int_{0}^1 f(t) \operatorname{d}t+\int_{0}^1 \frac{t(1-t) f^{\prime}(t)}{t-z} \operatorname{d}t .
\end{align*}
We return back to~\eqref{demon2}, and see that this is equivalent to prove that
\begin{align*}
\lim_{z \to 0}
z^2 (1-z) \left( \int_0^1 \frac{f(t)}{t-z}\, \d t\right)^\prime
= 0_N .
\end{align*}
This follows from the fact that the Stieltjes transform of $z(1-z)f^\prime(z)$, i.e.
\begin{align*}
\int_{0}^1 \frac{t(1-t) f^{\prime}(t)}{t-z} \operatorname{d}t ,
\end{align*}
is of the same type of $f$.
Then,
\begin{align*}
\lim_{z\to 0} z\left( \int_0^1 \frac{ t(1-t)f^\prime(t)}{t-z}\, \d t \right)
 =0_N, &&
\lim_{z\to 1} (1-z)\left( \int_0^1 \frac{t(1-t)f^\prime(t)}{t-z} \, \d t\right)
 =0_N ,
\end{align*}
and~\eqref{demon2} follows.

Now, as each entry of the matrix $Q_{n}^{\mathsf L}(z)$ is a Cauchy transform of certain function~$f$ described previously, by using~\eqref{demon1} and~\eqref{demon2} we have that,
\begin{align*} 
\left(Y_{n}^{\mathsf L}\right)^{\prime}(z)
=\left[\begin{matrix}
\operatorname{O}(1) & \operatorname{o}(\frac{1}{z^2}) \\
\operatorname{O}(1) & \operatorname{o}(\frac{1}{z^2})
\end{matrix}\right], 
 &&
\left(Y_{n}^{\mathsf L}(z)\right)^{-1}
=\left[\begin{matrix}
\operatorname{o}(\frac{1}{z}) & \operatorname{o}(\frac{1}{z}) \\
\operatorname{O}(1) & \operatorname{O}(1)
\end{matrix}\right],
&& z \to 0
\end{align*}
and 
\begin{align*}
\left(Y_{n}^{\mathsf L}\right)^{\prime}(z)
=\left[\begin{matrix}
\operatorname{O}(1) & \operatorname{o}(\frac{1}{(1-z)^2}) \\
\operatorname{O}(1) & \operatorname{o}(\frac{1}{(1-z)^2})
\end{matrix}\right], &&
\left(Y_{n}^{\mathsf L}(z)\right)^{-1}=\left[\begin{matrix}
\operatorname{o}(\frac{1}{1-z}) & \operatorname{o}(\frac{1}{1-z}) \\
\operatorname{O}(1) & \operatorname{O}(1)
\end{matrix}\right], &&
 z \to 1 .
\end{align*}
This implies that
\begin{align*}
\lim_{z \to 0}
z^{2} \left(Y_{n}^{\mathsf L}\right)^{\prime}(z)
\left(Y_{n}^{\mathsf L}\right)^{-1}
 & =\lim_{z \to 0} z^{2}
\left[\begin{matrix}
 \operatorname{o}(\frac{1}{z})+ \operatorname{o}(\frac{1}{z^2}) &\operatorname{o}(\frac{1}{z})+\operatorname{o}(\frac{1}{z^2}) \\
 \operatorname{o}(\frac{1}{z^2})+\operatorname{o}(\frac{1}{z}) & \operatorname{o}(\frac{1}{z^2})+ \operatorname{o}(\frac{1}{z})
\end{matrix}\right]
 \\ & 
=\lim_{z \to 0} z^{2}
\left[\begin{matrix}
 \operatorname{o}(\frac{1}{z^2}) &\operatorname{o}(\frac{1}{z^2}) \\
\operatorname{o}(\frac{1}{z^2}) &\operatorname{o}(\frac{1}{z^2})
\end{matrix}\right]
=0_{2 N} ,
\end{align*}
and
\begin{align*} 
\lim_{z \to 1} (1-z)^{2}\left(Y_{n}^{\mathsf L}\right)^{\prime}(z)\left(Y_{n}^{\mathsf L}\right)^{-1}=\lim_{z \to 1} (1-z)^{2}\left[\begin{matrix}
 \operatorname{o}(\frac{1}{(1-z)^2}) &\operatorname{o}(\frac{1}{(1-z)^2}) \\
\operatorname{o}(\frac{1}{(1-z)^2}) &\operatorname{o}(\frac{1}{(1-z)^2})
\end{matrix}
\right]
 = 0_{2 N} .
\end{align*}
Straightforward calculation and similar considerations lead us to
\begin{align*}
\lim_{z \to 0} z Y_{n}^{\mathsf L}(z)
\left[\begin{matrix}
h^{\mathsf L}(z) & 0_{N} \\
0_{N} & -h^{\mathsf R}(z)
\end{matrix}
\right]
\left(Y_{n}^{\mathsf L}(z)\right)^{-1} & =0_{2 N},
 \\
\lim_{z \to 1} (1-z) Y_{n}^{\mathsf L}(z)
\left[\begin{matrix}
h^{\mathsf L}(z) & 0_{N} \\
0_{N} & -h^{\mathsf R}(z)
\end{matrix}
\right]
\left(Y_{n}^{\mathsf L}(z)\right)^{-1} & =0_{2 N} .
\end{align*} 
Finally we arrive to
\begin{align*}
\lim_{z \to 0} z^{2} M_{n}^{\mathsf L}(z)=0_{2 N}
&&
\text{and}
&&
\lim_{z \to 1} (1-z)^{2} M_{n}^{\mathsf L}(z)=0_{2 N} .
\end{align*}
By analogous arguments we get the results for $M_n^{\mathsf R}$.
\end{proof}

\section{Differential relations from the Riemann--Hilbert problem} \label{sec:4}

Our objective is to derive differential equations satisfied by the biorthogonal matrix polynomials associated to regular Jacobi type matrices of weights.
Here we use the Riemann--Hilbert problem approach in order to derive these differential relations.

Let us define a new matrix functions,
\begin{align*}
\tilde{M}^\mathsf L_n (z)& = z(1-z){M}^\mathsf L_n (z) , &\tilde{M}^\mathsf R_n (z)& = z(1-z){M}^\mathsf R_n (z) ,
\end{align*}
then $\tilde{M}^\mathsf L_n (z)$ and $\tilde{M}^\mathsf R_n (z)$ are matrices of entire functions, cf. Theorem~\ref{prop:mn}.

\begin{pro}[First order differential equation for the fundamental matrices] \label{prop:mnnn}
In the conditions of Theorem~\ref{prop:mn} we have that 
\begin{align}
\label{eq:firstodeYnL}
z(1-z) \big(Y^{\mathsf L}_n\big)^{\prime} (z) + Y^{\mathsf L}_n (z)
\left[ \begin{matrix}
h^{\mathsf L} (z) & 0_N \\ 0_N & - h^{\mathsf R} (z)
 \end{matrix} \right] 
&
= \tilde{M}^{\mathsf L}_n (z) Y^{\mathsf L}_n (z)
 \\
\label{eq:firstodeYnR}
z(1-z) \big(Y^{\mathsf R}_n\big)^{\prime} (z)
 +
\left[ \begin{matrix}
h^{\mathsf R} (z) & 0_N \\ 0_N & - h^{\mathsf L} (z)
 \end{matrix} \right] 
Y^{\mathsf R}_n (z)
&
= Y^{\mathsf R}_n (z) \tilde{M}^{\mathsf R}_n (z) .
\end{align}
\end{pro}

\begin{proof}
Equations~\eqref{eq:firstodeYnL} and~\eqref{eq:firstodeYnR} follows immediately from the definition of the matrices~$M^{\mathsf L}_n(z)$ and $M^{\mathsf R}_n(z)$ in~\eqref{eq:Mn}. 
\end{proof}

\begin{pro}
In the conditions of Theorem~\ref{prop:mn}. If $h^{\mathsf L} (z)=A^{\mathsf L} z+B^{\mathsf L}$
and $ h^{\mathsf R}(z)=A^{\mathsf R} z+B^{\mathsf R} $, then the left and right fundamental matrices are given respectively by,
\begin{align}
\label{generalMnL}
 \tilde{M}^{\mathsf L}_n (z) & = 
\resizebox{.805\hsize}{!}{$
 \left[ \begin{matrix} 
\left(A^{\mathsf L}-nI_N\right) z +[ p^1_{\mathsf L,n}, A^{\mathsf L} ] + p^1_{\mathsf L,n}+ n I_N + B^{\mathsf L}
 &
A^{\mathsf L} C_n^{-1} + C_n^{-1} A^{\mathsf R}-(2n+1)C_{n}^{-1}
 \\[.15cm]
-C_{n-1} A^{\mathsf L} -A^{\mathsf R} C_{n-1}+(2n-1)C_{n-1}
&
\left(nI_N-A^{\mathsf R}\right) z + [ p^1_{\mathsf R,n} ,A^{\mathsf R} ]- p^1_{\mathsf R,n} - n I_N- B^{\mathsf R}
 \end{matrix} \right] 
 $},
 \\
\label{generalMnR}
\tilde{M}^{\mathsf R}_n (z) & =
\resizebox{.805\hsize}{!}{$
\left[ \begin{matrix}
\left(A^{\mathsf R}-nI_N\right) z - [ p^1_{\mathsf R,n} ,A^{\mathsf R} ]+ p^1_{\mathsf R,n} + n I_N+ B^{\mathsf R} &-C_{n-1} A^{\mathsf L} -A^{\mathsf R} C_{n-1}+(2n-1)C_{n-1} \\[.15cm]
 A^{\mathsf L} C_n^{-1} + C_n^{-1} A^{\mathsf R}-(2n+1)C_{n}^{-1}
 & \left(nI_N-A^{\mathsf L}\right) z -[ p^1_{\mathsf L,n}, A^{\mathsf L} ] - p^1_{\mathsf L,n}- n I_N - B^{\mathsf L} 
 \end{matrix} \right]
$} .
\end{align} 
\end{pro}

\begin{proof}
Taking $\vert z\vert \to +\infty $ in~\eqref{eq:numerar} we have that
\begin{align*}
Y_n^\mathsf{L}
=
\resizebox{.805\hsize}{!}{$
\left[ \begin{matrix}
I_{N} z^{n}+p_{\mathsf{L}, n}^{1} z^{n-1}+\cdots & -C_{n}^{-1}\left(I_{N} z^{-n-1}+q_{\mathsf{L}, n}^{1} z^{-n-2}+\cdots\right) \\[.15cm]
-C_{n-1}\left(I_{N} z^{n-1}+p_{\mathsf{L}, n-1}^{1} z^{n-2}+\cdots\right) & I_{N} z^{-n-2}+q_{\mathsf{L}, n-1}^{1} z^{-n-3}+\cdots
\end{matrix}
\right]
$}
 ,
 \\
\left(Y_n^\mathsf{L}\right)^{-1}
=
\resizebox{.805\hsize}{!}{$
\left[ \begin{matrix}
I_{N} z^{-n-2}+q_{\mathsf{R}, n-1}^{1} z^{-n-3}+\cdots 
&\left(I_{N} z^{-n-1}+q_{\mathsf{R}, n}^{1} z^{-n-2}+\cdots\right) C_{n}^{-1} \\
 \left(I_{N} z^{n-1}+p_{\mathsf{R}, n-1}^{1} z^{n-2}+p_{\mathsf{R}, n-1}^{2} z^{n-3}+\cdots \right)C_{n-1} & I_{N} z^{n}+p_{\mathsf{R}, n}^{1} z^{n-1}+p_{\mathsf{R}, n}^{2} z^{n-2}+\cdots
\end{matrix}\right]
$} 
 .
\end{align*}
Hence, as $\vert z\vert \to +\infty $
\begin{multline*}
z(1-z)\left(Y_{n}^{\mathsf L}\right)^{\prime}\left(Y_{n}^{\mathsf L}\right)^{-1}
 \\
=
\resizebox{.75\hsize}{!}{$
\left[ \begin{matrix}
-nI_Nz+nI_N-(nq_{\mathsf{R}, n-1}^{1}+(n-1) p_{\mathsf{L}, n}^{1}) & -(2n+1)C_{n}^{-1}\\
(2n-1)C_{n-1} & nI_Nz-nI_N +np_{\mathsf{R}, n}^{1}+(n+1)q_{\mathsf{L}, n-1}^{1}
\end{matrix}\right]
$}
+ \operatorname O({1}/{z}) .
\end{multline*}
Since $z\mapsto z(1-z)\left(Y_{n}^{\mathsf L}\right)^{\prime}(z)\left(Y_{n}^{\mathsf L}(z)\right)^{-1}$ is holomorphic over $\mathbb{C}$, by Liouville theorem we deduce that,
\begin{multline*}
z(1-z)\left(Y_{n}^{\mathsf L}\right)^{\prime}(z)\left(Y_{n}^{\mathsf L}(z)\right)^{-1}
 \\ =
\left[ \begin{matrix}
-nI_Nz+nI_N-(n q_{\mathsf{R}, n-1}^{1}+(n-1) p_{\mathsf{L}, n}^{1}) & -(2n+1)C_{n}^{-1}\\
(2n-1)C_{n-1} & nI_Nz-nI_N +np_{\mathsf{R}, n}^{1}+(n+1)q_{\mathsf{L}, n-1}^{1}
\end{matrix}\right] .
\end{multline*}
Using again Liouville's theorem we get,
\begin{multline*}
Y_{n}^{\mathsf L}(z)
\left[
\begin{matrix}
h^{\mathsf L}(z) & 0_{N} \\
0_{N} & -h^{\mathsf R}(z)
\end{matrix}\right]\left(Y_{n}^{\mathsf L}(z)\right)^{-1}
 \\ =
\left[\begin{matrix}
A^{\mathsf L}z+A^{\mathsf L}q_{\mathsf R,n-1}^1+p_{\mathsf L,n}^1A^{\mathsf L}+B^{\mathsf L} & A^{\mathsf L}C_{n}^{-1}-C_{n}^{-1}A^{\mathsf R}\\
-C_{n-1}A^{\mathsf L}+A^{\mathsf R}C_{n-1}& -A^{\mathsf R}z-A^{\mathsf R}p_{\mathsf R,n}^1-q_{\mathsf L,n-1}^1A^{\mathsf R}-B^{\mathsf R}
\end{matrix}\right] .
\end{multline*}
By considering the identities $p^1_{\mathsf R,n} = -q^1_{\mathsf L,n-1}$ and $p^1_{\mathsf L,n} = -q^1_{\mathsf R,n-1}$, then~\eqref{generalMnL} follows. The relation~\eqref{MnLR} leads to~\eqref{generalMnR}.
\end{proof}

Now, we introduce the $\mathcal N $ map, 
$
\displaystyle 
\mathcal N (F(z)) = F^{\prime}(z)+\frac{F^2(z)}{z(1-z)}
$.

\begin{pro}[Second order differential equation for the fundamental matrices] \label{prop:mnn}
In the conditions of Theorem~\ref{prop:mn} we have that
\begin{align} \label{eq:edo1}
\resizebox{.9095\hsize}{!}{$
z(1-z) \big(Y^\mathsf L_n\big)^{\prime\prime} 
+ \big(Y^\mathsf L_n\big)^{\prime} 
\left[ \begin{matrix} 
 2 h^{\mathsf L}+ (1-2z)I_N & 0_N \\ 
 0_N & - 2 h^{\mathsf R} +(1-2z) I_N
 \end{matrix} \right]
 + Y_n^\mathsf L (z)
 \left[ \begin{matrix}
 \mathcal N (h^{\mathsf L}) & 0_N \\ 
 0_N &\mathcal N (- h^{\mathsf R}) 
 \end{matrix} \right] 
= \mathcal N ( \tilde{M}^\mathsf L_n )Y^\mathsf L_n 
$} ,
 \\
\label{eq:edo2}
\resizebox{.9095\hsize}{!}{$
z(1-z) \big(Y^\mathsf R_n\big)^{\prime\prime} 
+
\left[ \begin{matrix} 
2 h^{\mathsf R} + (1-2z)I_N & 0_N\\ 
 0_N & - 2 h^{\mathsf L} + (1-2z)I_N 
 \end{matrix} \right] 
 \big(Y^\mathsf R_n\big)^{\prime} 
 + 
\left[ \begin{matrix} 
 \mathcal N (h^{\mathsf R}) & 0_N\\ 
 0_N & \mathcal N (-h^{\mathsf L}) 
 \end{matrix} \right] 
Y_n{^\mathsf R} (z)
 = 
Y^\mathsf R_n \mathcal N (\tilde{M}^\mathsf R_n )
$}.
 \end{align}
\end{pro}

\begin{proof}
Differentiating in~\eqref{eq:Mn} we get
\begin{align*}
\left(Z_{n}^{\mathsf L}\right)^{\prime \prime}\left(Z_{n}^{\mathsf L}\right)^{-1}=\frac{\left(\tilde{M}_{n}^{\mathsf L}\right)^{\prime}}{z(1-z)}-(1-2z)\frac{\tilde{M}_{n}^{\mathsf L}}{z^{2}(1-z)^2}+\frac{\left(\tilde{M}_{n}^{\mathsf L}\right)^{2}}{z^{2}(1-z)^2},
\end{align*}
so that
\begin{align*}
z(1-z)\left(Z_{n}^{\mathsf L}\right)^{\prime \prime}\left(Z_{n}^{\mathsf L}\right)^{-1}+(1-2z)M_n^{\mathsf L}=\left(\tilde{M}_{n}^{\mathsf L}\right)^{\prime}+\frac{\left(\tilde{M}_{n}^{\mathsf L}\right)^{2}}{z(1-z)}= \mathcal N (\tilde{M}^\mathsf L_n ) .
\end{align*}
Now let us see that
\begin{align*}
(1-2z)M_n^{\mathsf L}=z(1-z)\left(Y_n^{\mathsf L}\right)^\prime{(Y_n^{\mathsf L})}^{-1} + Y_n^{\mathsf L}\left[\begin{matrix}
h^{\mathsf L}&0_N
 \\
0_N & -h^{\mathsf R}
\end{matrix}
\right]
{(Y_n^{\mathsf L})}^{-1} .
\end{align*}
From~\eqref{eq:pearsonjacobileft} we have
\begin{align*}
z(1-z){\left(W^{\mathsf L}\right)}^{\prime\prime}{\left(W^{\mathsf L}\right)}^{-1}=\frac{{(h^{\mathsf L})}^2}{z(1-z)}-\frac{1-2z}{z(1-z)}h^{\mathsf L}+{(h^{\mathsf L})}^\prime ,
\end{align*}
and
\begin{align*}
z(1-z){\left((W^{\mathsf R})^{-1}\right)}^{\prime\prime}W^{\mathsf R}=\frac{({h^{\mathsf R})}^2}{z(1-z)}+\frac{1-2z}{z(1-z)}h^{\mathsf R}-{(h^{\mathsf R})}^\prime .
\end{align*}
Since
\begin{multline*}
z(1-z)\left(Z_{n}^{\mathsf L}\right)^{\prime \prime}{\left(Z_{n}^{\mathsf L}\right)}^{-1}=z(1-z){(Y_n^{\mathsf L})}^{\prime\prime}Y_n^{\mathsf L}+{\left(Y_n^{\mathsf L}\right)}^\prime
\left[\begin{matrix}
2h^{\mathsf L}& 0_N \\
0_N & -2h^{\mathsf R}
\end{matrix}
\right]
{(Y_n^{\mathsf L})}^{-1} \\
+Y_n^{\mathsf L}\left[\begin{matrix}
{z(1-z)\left(W^{\mathsf L}\right)}^{\prime\prime}{\left(W^{\mathsf L}\right)}^{-1} & 0_N \\
0_N & {z(1-z)\left((W^{\mathsf R})^{-1}\right)}^{\prime\prime}W^{\mathsf R} 
\end{matrix}\right]
{(Y_n^{\mathsf L})}^{-1} ,
\end{multline*}
we get the stated result~\eqref{eq:edo1}. The equation~\eqref{eq:edo2} follows in a similar way from definition of~$M_n^{\mathsf R}$ in~\eqref{eq:Mn}.
\end{proof}

We introduce the fol\-low\-ing $\C^{2N\times 2N}$ valued functions
\begin{align*}
\mathsf H_n^{\mathsf L} =
\left[ \begin{matrix}
 \mathsf H_{1,1,n}^{\mathsf L} & \mathsf H_{1,2,n}^{\mathsf L} \\[.05cm] 
 \mathsf H_{2,1,n}^{\mathsf L} & \mathsf H_{2,2,n}^{\mathsf L} 
 \end{matrix} \right]
 &
 : = 
 \mathcal N (\tilde{M}^\mathsf L_n ), &
\mathsf H_n^{\mathsf R} =
\left[ \begin{matrix}
\mathsf H_{1,1,n}^{\mathsf R} & \mathsf H_{1,2,n}^{\mathsf R} \\[.05cm]
 \mathsf H_{2,1,n}^{\mathsf R} & \mathsf H_{2,2,n}^{\mathsf R} 
 \end{matrix} \right]
&
:= \mathcal N (\tilde{M}^\mathsf R_n ).
\end{align*}
It holds~that the second order matrix differential equations~\eqref{eq:edo1} and~\eqref{eq:edo2} split in the fol\-low\-ing differential relations
\begin{align*}
z(1-z) \big(P_n^{\mathsf L} \big)^{\prime \prime} + \big( P_n^{\mathsf L} \big)^{ \prime}\big( 2 h^{\mathsf L} +(1-2z) I_N \big)
+
P_n^{\mathsf L} \mathcal N ( h^{\mathsf L}) 
& = \mathsf H_{1,1,n}^{\mathsf L} P_n^{\mathsf L} - \mathsf H_{1,2,n}^{\mathsf L} C_{n-1}P_{n-1}^{\mathsf L} , 
 \\
z(1-z) \big(Q_n^{\mathsf L} \big)^{\prime \prime}-\big(Q_n^{\mathsf L} \big)^{ \prime}\big( 2 h^{\mathsf R} - (1-2z)I_N \big)
+
Q_n^{\mathsf L} \mathcal N ( -h^{\mathsf R} ) & =
 \mathsf H_{1,1,n}^{\mathsf L} Q_n^{\mathsf L} - \mathsf H_{1,2,n}^{\mathsf L} C_{n-1} Q_{n-1}^{\mathsf L} ,
 \\
z(1-z) \big(P_n^{\mathsf R} \big)^{\prime \prime} + \big( 2 h^{\mathsf R} +(1-2z) I_N \big) \big(P_n^{\mathsf R} \big)^{ \prime}
+ \mathcal N ( h^{\mathsf R} ) P_n^{\mathsf R} 
 &= 
P_n^{\mathsf R} \mathsf H_{1,1,n}^{\mathsf R} 
- P_{n-1}^{\mathsf R} C_{n-1}\mathsf H_{2,1,n}^{\mathsf R} ,
 \\
z(1-z) \big(Q_n^{\mathsf R} \big)^{\prime \prime}- \big( 2 h^{\mathsf L} -(1-2z) I_N \big) \big(Q_n^{\mathsf R} \big)^{ \prime}
+ \mathcal N (- h^{\mathsf L}) Q_n^{\mathsf R} 
 &= 
Q_n^{\mathsf R} \mathsf H_{1,1,n}^{\mathsf R} 
- Q_{n-1}^{\mathsf R} C_{n-1}\mathsf H_{2,1,n}^{\mathsf R}.
\end{align*}
Using the calculation made in~\eqref{generalMnL} We want to recover here some known formulas in the scalar case.

\begin{nexam}
Let us consider the weight $W(z)=z^\alpha (1-z)^\beta$, with
 $\alpha$, $\beta$ scalars in $(-1, \infty)$. Then, the scalar second order equation for $\left\{P^{\mathsf L}_n\right\}_{n\in\N}$ and $\left\{Q^{\mathsf L}_n\right\}_{n\in\N}$ (cf. for example~\cite{Szego}) is given by
\begin{align} \label{scalarsode}
z(1-z) P_n^{\prime\prime} (z) + \big(1+\alpha - (\alpha+\beta +2)z\big) P_n^{\prime} (z) + n(\alpha+\beta+n+1) P_n (z)&=0 , \\ 
 \label{scalarsodeq}
z(1-z) Q_n^{\prime\prime} (z) + \big(1-\alpha + (\alpha+\beta -2)z\big) Q_n^{\prime} (z) +(n+1)(\alpha+\beta+n) Q_n (z)&=0 .
\end{align}
\end{nexam}

In fact, from~\eqref{generalMnL}
\begin{align*}
\tilde{M}^{\mathsf L}_n (z)=\left[ \begin{matrix}
-(\frac{\alpha+\beta}{2}+n)z+p_n^1+n+\frac{\alpha}{2}& -C_n^{-1}(\alpha+\beta+2n+1)\\
C_{n-1}(\alpha+\beta+2n-1)& (\frac{\alpha+\beta}{2}+n)z-p_n^1-n-\frac{\alpha}{2}
 \end{matrix} \right] ,
\end{align*}
it is easy to see that 
\begin{align*}
\left( \tilde{M}^{\mathsf L}_n (z)\right) ^2=
\left(-(\frac{\alpha+\beta}{2}+n)z+p_n^1+n+\frac{\alpha}{2}\right)^2 -\gamma_n\left((\alpha+\beta+2n)^2-1 \right)
I_2,
\end{align*}
and also,
$ \left( \tilde{M}^{\mathsf L}_n (z)\right)^\prime=\left[ \begin{matrix}
-\frac{\alpha+\beta}{2}+n&0\\ 0 &\frac{\alpha+\beta}{2}-n
 \end{matrix} \right] $.
Using now proposition~\ref{prop:mnn} we get
\begin{align*}
z&(1-z)P_n^{\prime\prime} (z) + \big(1+\alpha - (\alpha+\beta +2)z\big) P_n^{\prime} (z) -
n(\alpha+\beta+n+1) P_n (z)
\\
 &=\left(\frac{n(\alpha+n)+p_n^1(\alpha+\beta+2n)}{1-z}+\frac{\frac{\alpha^2}{4}-\left(p_n^1+\frac{\alpha}{2}+n \right)^2+\gamma_n\left((\alpha+\beta+2n)^2-1 \right)}{z(1-z)} \right)
 P_n (z) .
\end{align*}
By equalizing poles between left and right hand side on $0$ then on $1$ we have 
\begin{align*}
\left(\frac{\alpha^2}{4}-\left(p_n^1+\frac{\alpha}{2}+n \right)^2+\gamma_n\left((\alpha+\beta+2n)^2-1 \right)\right)P_n (0) & =0
 \\
\left(n(\alpha+n)+p_n^1(\alpha+\beta+2n)\right)P_n (1) &=0
\end{align*}
which, taking into account $P_n (0)$, $P_n (1)\neq 0$, leads to the representation of $p_n^1$ and~$\gamma_n$, as well as~\eqref{scalarsode}.
The equation~\eqref{scalarsodeq} for the $\{ Q_n\}_{n \in \mathbb N}$ follows from the above considerations.

\section{Matrix discrete Painlev\'e~IV} \label{sec:5}

We can consider, using the notation introduced before, the matrix weight measure $W (z) = W_{\mathsf L} (z) W_{\mathsf R} (z) $
such that
\begin{align*}
z(1-z) (W^\mathsf L)^{\prime}(z) &= ( h_0^{\mathsf L} +
h_1^{\mathsf L} z + h_2^{\mathsf L} z^2) W^{\mathsf L}(z),
 \\
z(1-z)(W^\mathsf R)^{\prime}(z)
 & = W^{\mathsf R}(z) (h_0^{\mathsf R} +
h_1^{\mathsf R} z + h_2^{\mathsf R} z^2) .
\end{align*}
From Theorem~\ref{prop:mn} we get the matrix $\tilde{M}_n=z(1-z) M_n^{\mathsf L} $ is given explicitly by
\begin{align*}
\begin{cases}
& (\tilde{M}_n^{\mathsf L})_{11} = C_{n}^{-1} h_2^{\mathsf R} C_{n-1} +
(h_0^{\mathsf L} + h_1^{\mathsf L} z + h_2^{\mathsf L} z^2) + h_1^{\mathsf L}
q_{\mathsf R,n-1}^{1} + p_{\mathsf L,n}^{1} h_1^{\mathsf L} \\
& \phantom{o}
+ z(h_2^{\mathsf L}q_{\mathsf R,n-1}^{1}+ p_{\mathsf L,n}^{1} h_2^{\mathsf L})
+ h_2^{\mathsf L} q_{\mathsf R,n-1}^{2}
+ p_{\mathsf L,n}^{2} h_2^{\mathsf L}
+ p_{\mathsf L,n}^{1} h_2^{\mathsf L} q_{\mathsf R,n-1}^{1}+nI_N-znI_N+p_{\mathsf L,n}^1 , 
\\
& (\tilde{M}_n^{\mathsf L})_{12} = (h_1^{\mathsf L} + h_2^{\mathsf L} z +
h_2^{\mathsf L} q_{\mathsf R,n}^{1} + p_{\mathsf L,n}^{1} h_2^{\mathsf L})
C_n^{-1} + C_n^{-1} (h_1^{\mathsf R} + h_2^{\mathsf R} z + h_2^{\mathsf R}
p_{\mathsf R,n}^{1} + q_{\mathsf L,n}^{1} h_2^{\mathsf R}) \\
& \phantom{okokokokokokokokokokokokokokokokokokokokokok}-(2n+1)C_{n}^{-1} ,
\\
& (\tilde{M}_n^{\mathsf L})_{21} = -C_{n-1} (h_1^{\mathsf L} + h_2^{\mathsf L} z +
h_2^{\mathsf L} q_{\mathsf R,n-1}^{1} + p_{\mathsf L,n-1}^{1} h_2^{\mathsf L}) \\
& \phantom{olaolaola}- (h_1^{\mathsf R} + h_2^{\mathsf R} z + h_2^{\mathsf R} p_{\mathsf
R,n-1}^{1} + q_{\mathsf L,n-1}^{1} h_2^{\mathsf R}) C_{n-1} +(2n-1)C_{n-1} ,
\\
 & (\tilde{M}_n^{\mathsf L})_{22} = -C_{n-1} h_2^{\mathsf L} C_n^{-1} -
(h_0^{\mathsf R} + h_1^{\mathsf R} z + h_2^{\mathsf R} z^2) - h_1^{\mathsf R}
p_{\mathsf R,n}^{1} - q_{\mathsf L,n-1}^{1} h_1^{\mathsf R} \\
&
\,
- z
(h_2^{\mathsf R}p_{\mathsf R,n}^{1}
+ q_{\mathsf L,n-1}^{1} h_2^{\mathsf R})
 - h_2^{\mathsf R} p_{\mathsf R,n}^{2}
- q_{\mathsf L,n-1}^{2} h_2^{\mathsf R}
- q_{\mathsf L,n-1}^{1} h_2^{\mathsf R} p_{\mathsf R,n}^{1} -nI_N +znI_N-p_{\mathsf R,{n}}^1 .
\end{cases}
\end{align*}
Using the three term recurrence relation for $\{ P_n^{\mathsf L} \}_{n \in \mathbb N}$ we get that 
$p_{\mathsf L,n}^1 - p_{\mathsf L,n+1}^1 = \beta_n^{\mathsf L}$ and $p_{\mathsf L,n}^2 - p_{\mathsf L,n+1}^2 = \beta_n^{\mathsf L} p_{\mathsf L,n}^1 + \gamma_n^{\mathsf L}$
where $\gamma_n^{\mathsf L} = C_{n}^{-1}C_{n-1} $.
Consequently,
\begin{align*}
p_{\mathsf L,n}^1 &= - \sum_{k=0}^{n-1} \beta_k^{\mathsf L} , &
p_{\mathsf L,n}^2 &=
\sum_{i,j=0}^{n-1} \beta_i^{\mathsf L} \beta_j^{\mathsf L} - \sum_{k=0}^{n-1}\gamma_k^{\mathsf L}.
\end{align*}
In the same manner, from the three term recurrence relation for $\{ Q_n^{\mathsf L} \}_{n \in \mathbb N}$ we deduce that
$q_{\mathsf L,n}^1 - q_{\mathsf L,n-1}^1 = \beta_n^{\mathsf R}:=C_n \beta_n^{\mathsf L} C_n^{-1}$ and $q_{\mathsf L,n}^2 - q_{\mathsf L,n-1}^2 = \beta_n^{\mathsf R} q_{\mathsf L,n}^1 + \gamma_n^{\mathsf R}$ ,
where $\gamma_n^{\mathsf R} = C_n C_{n+1}^{-1}$.

Now, we consider that $W = W^{\mathsf L}$ and $ W^{\mathsf R} = I_N$, and then use the representation for $ \{ P_n^{\mathsf L} \}_{n \in \N}$ and $\{ Q_n^{\mathsf L} \}_{n \in \N}$ in $z$ powers, the $(1,2)$ and $(2,2)$ entries in~\eqref{eq:firstodeYnL} read
\begin{align*} 
& (2n+1) (I_N-\beta_n^{\mathsf L}) + h_0^{\mathsf L} 
+ h_2^{\mathsf L} (\gamma_{n+1}^{\mathsf L} + \gamma_{n}^{\mathsf L} + (\beta_{n}^{\mathsf L})^2) + h_1^{\mathsf L} \beta_{n}^{\mathsf L} \\
&\phantom{okokokokokoko}= [p_{\mathsf L,n}^1, h_2^{\mathsf L}] p_{\mathsf L,n+1}^1
- [p_{\mathsf L,n}^2, h_2^{\mathsf L}] 
- [p_{\mathsf L,n}^1, h_1^{\mathsf L}]-p_{\mathsf L,n}^1-C_{n}^{-1}p_{\mathsf L,n+1}^1C_{n}, & \\
& \beta_n^{\mathsf L}-(\beta_n^{\mathsf L})^2= \gamma_{n}^{\mathsf L} \big( h_2^{\mathsf L}(\beta_n^{\mathsf L} + \beta_{n-1}^{\mathsf L}) + [p_{\mathsf L,n-1}^1, h_2^{\mathsf L}] + h_1^{\mathsf L}-(2n-1)I_N \big) \\
&\phantom{MMMMMM}
-
\big( h_2^{\mathsf L}(\beta_n^{\mathsf L} + \beta_{n+1}^{\mathsf L}) + [p_{\mathsf L,n}^1, h_2^{\mathsf L}] + h_1^{\mathsf L}-(2n+3)I_N \big) \gamma_{n+1}^{\mathsf L} - [p_{\mathsf L,n}^1, p_{\mathsf L,n+1}^1].
\end{align*}
We can write these equations as follows
\begin{align}\label{eq:dPIV_with_B_1}
& \hspace{-.225cm} (2n+1) I_N + h_0^{\mathsf L} 
+ h_2^{\mathsf L} (\gamma_{n+1}^{\mathsf L} + \gamma_{n}^{\mathsf L} )
+\left(h_2^{\mathsf L} \beta_{n}^{\mathsf L}+ h_1^{\mathsf L}-(2n+1)I_N\right) \beta_{n}^{\mathsf L} +\sum_{k=0}^{n-1} \beta_k^{\mathsf L} \\
& \hspace{-.225cm} \nonumber
 +C_{n}^{-1}\sum_{k=0}^{n} \beta_k^{\mathsf L}C_{n} = 
 \Big[ \sum_{k=0}^{n-1} \beta_k^{\mathsf L}, h_2^{\mathsf L}\Big] \sum_{k=0}^{n} \beta_k^{\mathsf L}
- \Big[\sum_{i,j=0}^{n-1} \beta_i^{\mathsf L} \beta_j^{\mathsf L} - \sum_{k=0}^{n-1}\gamma_k^{\mathsf L}
, h_2^{\mathsf L}\Big] - \Big[ \sum_{k=0}^{n-1} \beta_k^{\mathsf L}, h_1^{\mathsf L}\Big],
 \\
\label{eq:dPIV_with_B_2}
 &\hspace{-.225cm}
\beta_n^{\mathsf L}-(\beta_n^{\mathsf L})^2-\gamma_{n}^{\mathsf L} \big( h_2^{\mathsf L}(\beta_n^{\mathsf L} + \beta_{n-1}^{\mathsf L}) + h_1^{\mathsf L}-(2n-1)I_N \big) + \big( h_2^{\mathsf L}(\beta_n^{\mathsf L} + \beta_{n+1}^{\mathsf L}) + h_1^{\mathsf L} \\
 &\hspace{-.225cm} \nonumber
-(2n+3)I_N \big) \gamma_{n+1}^{\mathsf L} 
= \gamma_{n}^{\mathsf L} \Big[ \sum_{k=0}^{n-2} \beta_k^{\mathsf L} , h_2^{\mathsf L}\Big] -
\Big[\sum_{k=0}^{n-1} \beta_k^{\mathsf L} , h_2^{\mathsf L}\Big] \gamma_{n+1}^{\mathsf L} -\Big[\sum_{k=0}^{n-1} \beta_k^{\mathsf L} , \sum_{k=0}^{n} \beta_k^{\mathsf L}\Big] .
\end{align}
We will show now that this system contains a noncommutative version of an instance of discrete Painlev\'e~IV equation.

We see, on the \emph{r.h.s.} of the nonlinear discrete equations~\eqref{eq:dPIV_with_B_1} and~\eqref{eq:dPIV_with_B_2} nonlocal terms (sums) in the recursion coefficients $\beta_n^{\mathsf L}$ and $\gamma_n^{\mathsf L}$, all of them inside commutators. 
These nonlocal terms vanish whenever the three matrices $\{ h_0^{\mathsf L}, h_1^{\mathsf L}, h_2^{\mathsf L}\}$ conform an Abelian set, so that $\{ h_0^{\mathsf L}, h_1^{\mathsf L}, h_2^{\mathsf L},\beta_{n}^{\mathsf L}, \gamma_{n}^{\mathsf L}\}$ is also an Abelian set. In this commutative setting we~have
\begin{align*}
&
(2n+1) I_N + h_0^{\mathsf L} 
+ h_2^{\mathsf L} (\gamma_{n+1}^{\mathsf L} + \gamma_{n}^{\mathsf L} ))
+\left(h_2^{\mathsf L} \beta_{n}^{\mathsf L}+ h_1^{\mathsf L}-(2n+1)I_N\right) \beta_{n}^{\mathsf L}+p_{\mathsf L,n}^1+p_{\mathsf L,n+1}^1= 0_N,
 \\
&
\beta_n^{\mathsf L} -(\beta_n^{\mathsf L})^2 -\gamma_{n}^{\mathsf L} \big( h_2^{\mathsf L}(\beta_n^{\mathsf L} + \beta_{n-1}^{\mathsf L}) + h_1^{\mathsf L} -(2n-1)I_N\big) + \big( h_2^{\mathsf L}(\beta_n^{\mathsf L} + \beta_{n+1}^{\mathsf L}) \\
&
\hspace{8cm}
 + h_1^{\mathsf L}-(2n+3)I_N \big) \gamma_{n+1}^{\mathsf L} =0_N.
\end{align*}
In terms of
\begin{align*}
\xi_n := \frac{ h_0^{\mathsf L} }{2}+ n I_N+h_2^{\mathsf L} \gamma_n +p_{\mathsf L,n}^1
&&
\text{and}
&&
\mu_n := h_2^{\mathsf L} \beta_n ^{\mathsf L} +h_1^{\mathsf L}-(2n+1)I_N ,
\end{align*}
the above equations reads as
\begin{align*}
-\mu_n \beta^{\mathsf L}_n & = \xi_n + \xi_{n+1} 
&& \text{and} &&
\beta^{\mathsf L}_n (\xi_n -\xi_{n+1}) = 
 \mu_{n+1} \gamma_{n+1}-\gamma_n \mu_{n-1}.
\end{align*}
Now, we multiply the second equation by $\displaystyle \mu_n$ and taking into account the first one we arrive~to
\begin{align*}
- (\xi_n + \xi_{n+1})(\xi_n - \xi_{n+1})
= -\gamma_n \mu_{n-1} \mu_{n} + \gamma_{n+1} \mu_{n} \mu_{n+1} ,
\end{align*}
and so
\begin{align*}
 \xi_{n+1}^2 - \xi_n^2
= \gamma_{n+1} \mu_{n} \mu_{n+1} -\gamma_n \mu_{n-1} \mu_{n} .
\end{align*}
Hence,
\begin{align*}
\xi_{n+1}^2 - \xi_0^2 = \gamma_{n+1} \mu_{n} \mu_{n+1} 
&& \text{and} &&
\beta^{\mathsf L}_n \mu_n& = - (\xi_n + \xi_{n+1})
\end{align*}
coincide to the ones 
presented in~\cite{boelen_Van_Assche} as discrete Painlev\'e~IV (dPIV) equation.
In fact, taking $\nu_n = \mu_n^{-1}$ we finally arrive to
\begin{align*}
 \nu_{n} \nu_{n+1} 
 & = \frac{h_2^{\mathsf L}\big(\xi_{n+1} - h_0^{\mathsf L}/2 - nI_N-p_{\mathsf L,n}^1\big)}{\xi_{n+1}^2 - \xi_0^2} , \\
\xi_n + \xi_{n+1} & = 
\left(\left(h_2^{\mathsf L}\right)^{-1} h_1^{\mathsf L} -\left(h_2^{\mathsf L}\right)^{-1} \nu_n^{-1} -(2n+1)\left(h_2^{\mathsf L}\right)^{-1}\right) \nu_n^{-1} .
\end{align*}

Now, we are able to state that,

\begin{teo}[\textbf{Non-Abelian extension of the dPIV}]
Equations~\eqref{eq:dPIV_with_B_1}
and
\eqref{eq:dPIV_with_B_2}
defines a nonlocal nonlinear non-Abelian system for the recursion coefficients.
\end{teo}

\end{document}